\documentclass{amsart} 
\usepackage{graphicx}
\usepackage{amsfonts}
\usepackage{amsthm}
\usepackage{amsmath}
\usepackage{amssymb}
\usepackage{latexsym}
\usepackage[all]{xy}

\swapnumbers
 
\newcounter{zlist}
\newenvironment{zlist}{\begin{list}{\rm(\arabic{zlist})}{
\usecounter{zlist}\leftmargin2.5em\labelwidth2em\labelsep0.5em
\topsep0.6ex\itemsep0.3ex plus0.2ex minus0.3ex
\parsep0.3ex plus0.2ex minus0.1ex}}{\end{list}}
 
\newcounter{blist}
\newenvironment{blist}{\begin{list}{{\rm(\alph{blist})}}{
\usecounter{blist}\leftmargin2.5em\labelwidth2em\labelsep0.5em
\topsep0.6ex \itemsep0.3ex plus0.2ex minus0.3ex
\parsep0.3ex plus0.2ex minus0.1ex}}{\end{list}}
 
\newcounter{rlist}
\newenvironment{rlist}{\begin{list}{{\rm(\roman{rlist})}}{
\usecounter{rlist}\leftmargin2.5em\labelwidth2em\labelsep0.5em
\topsep0.6ex\itemsep0.3ex plus0.2ex minus0.3ex
\parsep0.3ex plus0.2ex minus0.1ex}}{\end{list}}

\newcommand{\lra}{\longrightarrow}
\newcommand{\LRa}{\Leftrightarrow}

\newcommand{\Mor}{{\rm Mor}}

\newcommand{\Set}{\textsf{Set}}

\newcommand{\Ra}{\Rightarrow}

\newcommand{\wF}{\widehat F}

\newcommand{\wL}{{\widehat L}}

\newcommand{\wR}{\widehat R}

\newcommand{\wT}{{\widehat T}}
\newcommand{\wtau}{{\widehat\tau}}
\newcommand{\wmu}{{\widehat\mu}}

\newcommand{\weta}{\widehat \eta}
\newcommand{\wve}{\widehat \ve}

\newcommand{\ceta}{\check{\eta}}
\newcommand{\cgamma}{\check{\gamma}}
\newcommand{\cmu}{\check{\mu}}
\newcommand{\cvartheta}{\check{\vartheta}}
\newcommand{\cdelta}{\check{\delta}}
\newcommand{\cve}{\check{\varepsilon}}

\newcommand{\wdelta}{\widehat{\delta}}

\newcommand{\cC}{\mathcal{C}}

\newcommand{\odelta}{\overline{\delta}}

\newcommand{\tK}{{\widetilde K}}
\newcommand{\tL}{{\widetilde L}}

\newcommand{\ul}{\underline}

\newcommand{\A}{\mathbb{A}}
\newcommand{\uA}{\underrightarrow{\mathbb{A}}}
\newcommand{\uB}{\underrightarrow{\mathbb{B}}}
\newcommand{\rA}{\underline{\mathbb{A}}}

\newcommand{\rB}{\underline{\mathbb{B}}}
\newcommand{\uM}{\underrightarrow{\mathbb{M}}}
\newcommand{\rM}{\underline{\mathbb{M}}}
\newcommand{\B}{\mathbb{B}}

\newcommand{\M}{\mathbb{M}}

\newcommand{\oeta}{\overline{\eta}}
\newcommand{\ove}{\overline{\varepsilon}}

\newcommand{\wkappa}{\widehat{\kappa}}

\newcommand{\omu}{\overline{\mu}}

\newcommand{\ot}{\otimes}

 \newcommand{\ueps}{{\underline\varepsilon}}
\newcommand{\uDelta}{\underline\Delta}

\newcommand{\ve}{\varepsilon}

\newcommand{\up}{\upsilon}

\newcommand{\id}{I}

 \newcommand{\se}{_{\underline 1}}
 \newcommand{\sz}{_{\underline 2}}

\newcommand{\oG}{{\overline G}} 

\newcommand{\oT}{{\overline T}}

\newcommand{\dcirc}{\cdot}

 \swapnumbers

\newtheorem{theorem}{Theorem}[section]

\newtheorem{thm}[theorem]{}

\numberwithin{equation}{section}

\newcommand{\btm}{\begin{thm}}
\newcommand{\etm}{\end{thm}}

\begin{document} 

\title[Regular pairings]{Regular pairings of functors and weak (co)monads}
 \author[Robert Wisbauer]{ Robert Wisbauer}  
\address{Department of Mathematics\\ 
 Heinrich Heine University D\"usseldorf, Germany \newline
e-mail: wisbauer@math.uni-duesseldorf.de}  
 
\begin{abstract}
For functors $L:\A\to \B$ and $R:\B\to \A$ between any  
categories $\A$ and $\B$,
a  {\em  pairing} is defined by maps, natural in $A\in \A$ and  $B\in \B$,
$$\xymatrix{\Mor_\B (L(A),B) \ar@<0.5ex>[r]^{\alpha} & 
 \Mor_\A (A,R(B))\ar@<0.5ex>[l]^{\beta}}.$$ 
 
$(L,R)$ is an {\em adjoint pair} provided $\alpha$ (or $\beta$) is a bijection.
In this case the composition $RL$ defines a monad on the category $\A$,  $LR$ defines a comonad
on the category $\B$, and there is a well-known correspondence between monads (or comonads)
and adjoint pairs of functors. 

For various applications it was observed that the conditions for a unit 
of a monad was too restrictive and weakening it still allowed for 
a useful generalised notion of a monad. This led to the introduction of 
{\em weak monads} and {\em weak comonads} and the definitions needed were made
 without referring to this kind of adjunction. 
The motivation for the present paper is to show that these notions can be naturally derived from 
 pairings of functors $(L,R,\alpha,\beta)$ with  $\alpha = \alpha\dcirc \beta\dcirc \alpha$ and
$\beta = \beta \dcirc\alpha\dcirc\beta$.  
Following closely the constructions known for monads (and unital modules) and 
comonads (and counital comodules), 
 we show that any weak (co)monad on $\A$ gives rise 
to a regular pairing between $\A$ and the category of {\em compatible (co)modules}.          
\smallskip 


\end{abstract}
\maketitle

\vspace{-0.5cm}

\tableofcontents
\vspace{-0.5cm}

\section{Introduction}

Similar to the unit of an algebra, the existence of a unit of a monad is 
essential for (most of) the interesting properties of the related structures.  
Yet, there are numerous applications for which the request for a unit 
of a monad is too restrictive. Dropping the unit completely makes the 
theory fairly poor and the question was how to weaken 
the conditions on a unit such that still an effective theory can be developped.  
The interest in these questions was revived, for example, by the study of {\em weak Hopf algebras}
 by G. B\"ohm et al. in \cite{BS} and {\em weak entwining structures} by 
 S. Caenepeel et al. in \cite{CaeGro} (see also \cite{AFGR}, \cite{BW}). 
To handle this situation the theory of weak monads and comonads was developped 
and we refer to  \cite{B-weak} for a recent account on this theory. 
  
On any category, monads are induced by a pair of adjoint functors
and, on the other hand, any monad $(F,\mu,\eta)$ induces an adjoint pair of functors,
the free functor $\phi_F:\A\to \A_F$ and the forgetful functor $U_F:\A_F\to \A$,
where $\A_F$ denotes the catgeory of unital $F$-modules. 
This is all shown in Eilenberg-Moore \cite{EM}. 

In this correspondence the unitality of the monad is substantial
and the purpose of the present paper is
to exhibit a similar relationship between 
weak (co)monads and generalised forms of adjunctions. 
To this end, for functors $L:\A\to \B$ and $R:\B\to \A$ between
categories $\A$ and $\B$,
we consider maps 
$$\xymatrix{\Mor_\B (L(A),B) \ar@<0.5ex>[r]^{\alpha } & 
 \Mor_\A (A,R(B))\ar@<0.5ex>[l]^{\beta}},$$ 
required to be natural in $A\in \A$ and  $B\in \B$.
We call this a {\em pairing of functors}, 
or a {\em full pairing} if we want to stress that we have maps in both directions.
 Such a pairing is said to be {\em regular} provided $\alpha$ and $\beta$ are regular maps, 
 more precisely, 
\begin{center}
 $\alpha = \alpha\dcirc \beta\dcirc \alpha$\quad and \quad
$\beta = \beta \dcirc\alpha\dcirc\beta$. 
\end{center}

In Section \ref{adj-con},  regular pairings of functors are defined and some of their
 general properties are described.

 Motivated by substructures showing up in pairings of funcoter, in Section \ref{quasi-mon},
 {\em $q$-unital monads} $(F,\mu,\eta)$ on $\A$ are defined as endofunctors
$F:\A\to \A$ with natural transformations
$\mu:FF\to F$ and $\eta:\id_\A\to F$ ({\em quasi-unit}) and the sole condition that 
$\mu$ is associative.  {\em (Non-unital) $F$-modules} are defined by morphisms
$\varrho:F(A)\to A$ satisfying  $\varrho\circ\mu = \varrho\circ F\varrho$,
and the category of all $F$-modules is denoted by $\uA_F$.
For these data the free and forgetful functors,
 \begin{center}
$\phi_F:\A \to \uA_F$ \quad and \quad $U_F:\uA_F\to \A$.
\end{center}
give rise to a full pairing. From this we define {\em regularity} 
of $\eta$ and {\em compatibility} for the $F$-modules.   
The $q$-unital monad $(F,\mu,\eta)$ is said to be 
{\em $r$-unital} (short for {\em regular-unital}) 
 provided  $\eta$ is regular and $\mu$ is compatible as an $F$-module.
Now the free functor $\phi_F: \A\to \rA_F$ with the forgetful functor $U_F:\rA_F\to \A$
form a regular pairing, where $\rA_F$ denotes the (sub)category of compatible $F$-modules. 

The dual notions for (non-counital) comonads are outlined in Section \ref{qu-comon} and 
at the end of the section the comparison functors for a regular pairing $(L,R,\alpha,\beta)$
are considered (see \ref{cont-comon}).

In Section \ref{entw-qu-mon} we study the lifting of functors between categories to the 
corresponding categories of compatible modules or compatible comodules, respectively.
This is described by generalising Beck's {\em distributive laws} (see \cite{Beck}), 
also called {\em entwinings}, 
and it turns out that most of the diagrams are the same as for 
the lifting to unital modules (e.g. \cite{W-CoAl})  but to compensate the missing unitality
extra conditions are imposed on the entwining natural trans\-formation
(e.g. Proposition \ref{nat.trans.p}).
In this context we obtain a generalisation of Applegate's lifting theorem for (co)monads to 
weak (co)monads (Theorem \ref{reg-lift}, \ref{reg-lift-co}). 

Lifting an endofunctor $T$ of $\A$ to an endofunctor $\oT$ of $\rA_F$ leads to the 
question when $\oT$ is a weak monad  ($TF$ allows for the structure of a weak monad)
 and in Section \ref{lift-e-qu-mod} we provide conditions to make this happen.
 
The final Section \ref{mixed-entw} is concerned  with  weak monads  
 $(F,\mu,\eta)$ and weak comonads $(G,\delta,\ve)$ on any category $\A$
  and the interplay between the respective lifting properties.  
 Hereby properties of the lifting $\oG$ to $\rA_F$ and the lifting
$\wF$ to $\rA^G$ are investigated (see Theorems \ref{oG} and \ref{wF}) which 
generalise observations known for weak bi-algebras (and weak Hopf algebras).

 In our setting, notions like {\em pre-units}, {\em pre-monads}, {\em weak monads}, 
{\em demi-monads}, {\em pre-$A$-corings}, {\em weak corings},  {\em weak Hopf algebras} from the 
 literature (e.g.  \cite{AFGR}, \cite{B-weak}, \cite{Brz:str},   \cite{BoLaSt}, 
\cite{Wi-weak}) find their natural environment.  

In the framework of 2-categories weak structures are investigated by B\"ohm et al. in 
\cite{B-weak}, \cite{BoLaSt} and an extensive list of examples of weak structures 
is given there.

\section{Pairings of functors}\label{adj-con}

Throughout $\A$ and $\B$ will denote arbitrary categories. 
By $\id_A$, $A$ or
just by $\id$, we denote the identity morphism of an object $A\in \A$,
$\id_F$ or $F$ stands for the identity natural transformation on the functor $F$,
and $\id_\A$ means the identity functor of a category $\A$. 
 We write $F_{-,-}$ for the natural transformation of bifunctors determined by the
maps  $F_{A,A'}:\Mor_\A(A,A')  \to \Mor_\B(F(A),F(A'))$ for $A,A'\in \A$.

Before considering regularity for natural transformations we recall basic 
properties of 
 
\btm \label{reg-mor} {\bf Regular morphisms.} \em
Let $A,A'$ be any objects in a category $\A$. Then a 
morphism $f:A\to A'$ is called {\em regular} provided there is 
a morphism $g:A'\to A$ with  $fgf=f$. 
Clearly, in this case $gf:A\to A$  and $fg:A'\to A'$ 
are idempotent endomorphisms. 

Such a morphism $g$ is not necessarily unique. In particular, for $gfg$
we also have 
$ f(gfg)f=fgf =f$, and the identity $(gfg)f(gfg)= gfg$ shows that  
$gfg$ is again a regular morphism. 

If idempotents split in $\A$, then every idempotent morphism $e:A\to A$  
 determines a subobject of $A$, we denote it by $eA$. 

If $f$ is regular with $fgf=f$, then the restriction of $fg$ is the identity morphism
on $fgA'$ and $gf$ is the identity on $gfA$. 

Examples for regular morphisms are retractions, coretractions, and isomorphisms.
For modules $M,N$ over any ring, a morphism $f:M\to N$ is regular if and only if 
the image and the kernel of $f$ are direct summands in $N$ and $M$, respectively. 

This notion of regularity is derived from von Neumann regularity 
of rings. For modules (and in preadditive categories)
it was considered by Nicholson, Kasch, Mader and others
(see \cite{KaMa}).
We use the terminology also for natural transformations and functors with 
obvious interpretations.
\etm

\begin{thm}\label{setting}{\bf Pairing of functors.} \em (e.g. \cite[2.1]{Par})  
Let $L:\A\to \B$ and $R:\B\to \A$ be covariant functors. 
Assume there are morphisms, natural in $A\in \A$ and $B\in \B$,
 $$\begin{array}{l}
\alpha :\Mor_\B (L(A),B) \to \Mor_\A (A,R(B)),\\[+1mm]  
\beta :\Mor_\A (A,R(B)) \to \Mor_\B (L(A),B).
\end{array}$$ 
These maps correspond to natural transformations  
  between functors $\A^{op}\times \B \to \Set$. 
 The quadruple $(L,R,\alpha, \beta)$ is called a  
{\em (full) pairing (of functors)}. 

Given such a pairing, the morphisms, 
for $A\in \A$, $B\in \B$,   
\begin{center}
 $\eta_A:= \alpha_{A,L(A)}(\id): A\to RL(A)$ \; and \; 
 $\ve_B:=  \beta_{R(B),B} (\id) : LR(B)\to B $
\end{center}
correspond to natural transformations 
$$\eta:\id_\A \to RL, \quad \ve: LR\to \id_\B,$$
which we call {\em quasi-unit} and {\em quasi-counit} of 
$(L,R,\alpha, \beta)$, respectively. 
 
From these the transformations $\alpha$ and $\beta$ are obtained by  
$$\begin{array}{rrcl}
\alpha_{A,B}:&  L(A)\stackrel{f}\lra B  &\longmapsto & 
A\stackrel{\eta_A}\lra  RL(A)\stackrel{R(f)}\lra R(B) , \\[+1mm]
\beta_{A,B}: &  A\stackrel{g}\lra R(B)  & \longmapsto &
 L(A)\stackrel{L(g)}\lra LR(B)\stackrel{\ve_B}\lra B .
\end{array} $$
Thus the  pairing $(L,R,\alpha, \beta)$ is also described by the 
quadruple $(L,R,\eta, \ve)$.
\medskip

Naturality of $\ve$  and $\eta$ induces    
an associative product  and a quasi-unit for the endofunctor $RL:\A\to\A$,
$$ R\ve L: RLRL \to RL,\quad \eta: I_\A\to RL,$$
and   a coassociative  
  coproduct  and a quasi-counit for the endofunctor $LR:\B\to \B$,
$$  L\eta R: LR\to LRLR,  \quad \ve: LR\to I_\B. $$
\end{thm}

By the Yoneda Lemma we can describe compositions of $\alpha$ and $\beta$ by the images of the 
identity transformations of the respective functors. 

\btm\label{a-b-comp}{\bf Composing $\alpha$ and $\beta$.} \em
Let $(L,R,\alpha, \beta)$ be a  pairing with quasi-unit $\eta$ and 
quasi-counit $\ve$.
The descriptions of $\alpha$ and $\beta$ in \ref{setting} yield,  for the 
identity transformations $I_L:L\to L$, $I_R:R\to R$, 
$$\begin{array}{rcl} 
\alpha(\id_L)& = & \id_\A \stackrel{\eta}\lra RL, \\ 
\beta\dcirc\alpha(\id_L) & = & 
 L\stackrel{L\eta}\lra LRL \stackrel{\ve L}\lra L, \\
\alpha\dcirc \beta\dcirc\alpha(\id_L)& = & 
\id_\A \stackrel{\eta}\lra RL
  \stackrel{RL\eta}\lra RLRL \stackrel{R\ve L}\lra RL , \\
\end{array}$$
 $$\begin{array}{rcl} 
\beta(\id_R)& = & LR \stackrel{\ve}\lra \id_\B, \\ 
\alpha\dcirc\beta(\id_R) & = & R \stackrel{\eta R}\lra RLR \stackrel{R\ve }\lra R, \\
\beta\dcirc \alpha\dcirc\beta(\id_R) & = & 
LR \stackrel{L\eta R}\lra LRLR \stackrel{LR\ve }\lra LR \stackrel{\ve}\lra \id_\B. 
\end{array}
$$
\end{thm}

The following morphisms will play a special role in what follows. 

\begin{thm}\label{nat-trans}{\bf Natural endomorphisms.}  \em
With the notions from {\rm \ref{setting}}, 
we define the natural transformations
$$\begin{array}{ll}
\vartheta:=R(\beta\alpha(I_L)):& \xymatrix{RL\ar[r]^{RL\eta\quad} & RLRL \ar[r]^{\quad R\ve L} & RL ,} \\
\ul{\vartheta}:= \alpha\beta ( R(I_L)):&\xymatrix{RL\ar[r]^{\eta RL\quad} & RLRL \ar[r]^{\quad R\ve L} & RL,}\\[+1mm] 
\gamma:=L(\alpha\beta(I_R)): & \xymatrix{LR \ar[r]^{L\eta R\quad} & LRLR \ar[r]^{\quad LR\ve } & LR ,} \\
\ul{\gamma}:= \beta\alpha(L(I_R)): & \xymatrix{LR \ar[r]^{L\eta R\quad} & LRLR \ar[r]^{\quad \ve LR} & LR,}  
\end{array}$$ 
which have the properties 
$$\begin{array}{ccc}
 R\ve L\dcirc RL\vartheta = \vartheta \dcirc R\ve L, &
   R\ve L\dcirc  \ul{\vartheta} RL = \ul{\vartheta}\dcirc R\ve L, &
 \ul{\vartheta} \dcirc \vartheta =\vartheta \dcirc \ul{\vartheta};  \\[+1mm]
  LR\gamma \dcirc L\eta R = L\eta R\dcirc \gamma, &
   \ul{\gamma} LR \dcirc L\eta R = L\eta R\dcirc \ul{\gamma}, &
 \ul{\gamma} \dcirc \gamma =\gamma \dcirc \ul{\gamma}.
\end{array} $$
\end{thm}
\smallskip 

\btm\label{a-regular}{\bf Definitions.} \em       %\label{b-regular}
Let $(L,R,\alpha,\beta)$ be a pairing (see {\rm \ref{setting}}). We call
 \medskip
 
\begin{tabular}{rcl}
  $\alpha$ {\em regular} & if &  % $\alpha\dcirc\beta\dcirc\alpha(I_L)=\alpha(I_L)$; 
     $\alpha\dcirc\beta\dcirc\alpha=\alpha$;              \\ 
 $\alpha$ {\em symmetric} & if &   $\vartheta = \ul{\vartheta}$.\\[+1mm] 
% $ R(\beta\alpha(I_L))= \alpha\beta ( R(I_L))$, i.e.,  $  
  $\beta$  {\em regular} & if & %$\beta\dcirc \alpha\dcirc\beta(I_R)=\beta(I_R)$ 
     $\beta\dcirc \alpha\dcirc\beta=\beta$;  \\[+1mm]   
 $\beta$  {\em symmetric} & if & $\gamma = \ul{\gamma}$; \\
%$\alpha\beta ( L(I_R))=L(\alpha\beta(I_R))$, i.e.,
  $(L,R,\alpha,\beta)$  {\em regular} % pairing} provided
  &if& $\alpha = \alpha\dcirc \beta\dcirc \alpha$ 
and $\beta = \beta \dcirc\alpha\dcirc\beta$.
\end{tabular}
\medskip

The following properties are easy to verify:

\begin{rlist}
\item If $\alpha$ is regular, then 
$\beta\dcirc\alpha(\id_L) $, $\vartheta$ 
 and $\ul{\vartheta}$ are idempotent and 
$\vartheta\dcirc \eta = \eta = \ul{\vartheta} \dcirc\eta$; \\
 furthermore, for $\beta':= \beta\dcirc\alpha\dcirc \beta$,  
$(L,R,\alpha,\beta')$ is a regular pairing. 

\item If $\beta$ is regular, then 
$\alpha\dcirc\beta(\id_R)$, $\gamma$ and $\ul{\gamma}$ are idempotent
 and $\ve\dcirc \gamma=\ve = \ve \dcirc \ul{\gamma}$; \\
 furthermore, for $\alpha':=\alpha\dcirc\beta\dcirc \alpha$,
 $(L,R,\alpha',\beta)$ is a regular pairing. 
\end{rlist}
\etm 

Any pairing $(L,R,\alpha,\beta)$ 
 with $\beta\dcirc \alpha=\id$ or $\alpha \dcirc \beta =\id$  is regular.  
 The second condition de\-fines the {\em semiadjoint functors} in Medvedev \cite{Med}. 
\smallskip

With manipulations known from ring theory one can show how 
  pairings with regular components can be related  with  adjunctions 
 provided idempotents split.

\btm\label{rel-adj}{\bf Related adjunctions.} \em 
Let $(L,R,\alpha,\beta)$ be a pairing
 (with quasi-unit $\eta$, quasi-counit $\ve$) and assume $\alpha$ to be regular.
 
 If the idempotent
 $h:= \beta\cdot \alpha(I_L): L\stackrel{L\eta}\lra LRL\stackrel{\ve L}\lra L$ splits, that is, 
there are a functor $\ul L: \A\to \B$ and natural transformations
 \begin{center} 
  $p:L\to \ul L$,\quad  $i:\ul L\to L$ \quad
with \quad $i\dcirc p=h$ \; and \; $p\dcirc i= \id_{\ul L}$, 
\end{center}
then the natural transformations  
$$\ul \eta:\xymatrix{\id_\A \ar[r]^\eta &RL\ar[r]^{R p}& R\ul L} ,\quad
\ul\ve: \xymatrix{\ul L R \ar[r]^{i R}&LR\ar[r]^{\ve} & \id_\B},
$$
 as quasi-unit and quasi-counit, define a  pairing 
$(\ul L, R,\ul\alpha, \ul\beta)$ with 
$\ul \beta\dcirc\ul \alpha=\id$.
 
If $\alpha\dcirc\beta= I$, then $(\wL, R,\ul\alpha, \ul \beta)$ is an adjunction. 
\medskip
 
In case the natural transformation $\beta$ is regular, similar constructions apply if we assume that the idempotent 
 $\alpha\cdot\beta(I_R): R\stackrel{\eta R}\lra RLR\stackrel{R\ve }\lra R$ splits.
\etm

The properties of the $(RL,R\ve R\eta)$ and $(LR,L\eta R,\ve)$ mentioned in \ref{setting}
motivate the defi\-nitions in the next section.

\section{Monads and modules}

\btm \label{quasi-mon} {\bf $q$-unital monads and their modules.} \em %{\em non-unital monad} 
We call $(F,\mu)$ a {\em functor with product} (or {\em non-unital monad})    
provided $F:\A\to \A$ is an endofunctor on a category $\A$ and 
$\mu: FF\to F$  is a natural transformation 
satisfying the associativity condition  $\mu \dcirc F\mu  = \mu \dcirc \mu F$.

For $(F,\mu)$, a {\em (non-unital) $F$-module} is defined as an object $A\in \A$ with a morphism
$\varrho: F(A)\to A$  in $\A$ satisfying $\varrho\dcirc F\varrho = \varrho\dcirc \mu_A$.

{\em Morphisms} between $F$-modules $(A,\varrho)$, 
$(A',\varrho')$ 
 are morphisms $f:A\to A'$ in $\A$ with 
      $\varrho'\dcirc F(f)= f\dcirc \varrho$. 
The set of all these is denoted by $\Mor_F(A,A')$.
With these morphisms, (non-unital) $F$-modules form a category which we denote by $\uA_F$.

By the associativity condition on $\mu$, for every $A\in \A$, $(F(A), \mu_A)$
is an $F$-module and this leads to the free functor and the  forgetful functor,
  $$\phi_F: \A\to \uA_F, \quad A \mapsto (F(A), \mu_A), \qquad
  U_F: \uA_F\to \A, \quad  (A,\varrho)  \mapsto A. $$
 
A triple $ (F,\mu,\eta)$ is said to be a {\em $q$-unital monad} on $\A$ provided
$(F,\mu)$ is a functor with product
and $\eta:\id_\A\to F$ is any natural transformation, called a {\em quasi-unit} (no
 additional properties are required). 
 One always can define natural transformations
$$\vartheta: F\stackrel{F\eta}\lra FF \stackrel{\mu}\lra F,
\quad \ul{\vartheta}: F\stackrel{\eta F}\lra FF \stackrel{\mu}\lra F .$$ 
Note that for any $A\in \A$, $\vartheta_A$ is in $\A_F$ and $\ul\vartheta_A$ 
is not necessarily so. 

Given $q$-unital monads $ (F,\mu,\eta)$, $(F',\mu',\eta')$ on $\A$,
 a natural transformation $h: F\to F'$ is called a {\em morphism of $q$-unital monads} if 
$$\mu'\dcirc hh = h\dcirc \mu\; \mbox{ and } \; \eta'=h\dcirc \eta. $$
 \etm

The existence of a quasi-unit allows the following generalisation of the Eilenberg-Moore 
construction for (unital) monads. 
 
\btm \label{mon-cont} {\bf  $q$-unital monads and pairings.} \em
  For a q-unital monad $(F,\mu,\eta)$ 
 we obtain a  pairing $(\phi_F,U_F,\alpha_F,\beta_F)$ with the maps,
for $A\in \A$,  $(B,\varrho)\in \uA_F$,
 $$\begin{array}{ll}
\alpha_{F}: \Mor_F(\phi_F(A),B)\to \Mor_\A(A,U_F(B)), & f\mapsto f\dcirc \eta_A,\\[+1mm]
 \beta_{F}:\Mor_\A(A,U_F(B))\to\Mor_F(\phi_F(A),B),& g\mapsto \varrho  \dcirc F(g).
\end{array}$$

The quasi-unit  $\eta$ is called {\em regular} if $\alpha_F$ is regular, that is,
$$\id_\A \stackrel{\eta}\lra F = 
\id_\A \stackrel{\eta}\lra F \stackrel{F\eta}\lra FF \stackrel{\mu}\lra F,$$
and we say  $\eta$ is {\em symmetric} if $\alpha_F$ is so, that is, $\vartheta=\ul\vartheta$.  
 \smallskip

An $F$-module $\varrho:F(A)\to A$ in $\uA_F$ is said to be {\em compatible}  if
$\beta_F\alpha_F(\varrho)=\varrho$, that is   
$$ F(A)\stackrel{\varrho}\lra A =
F(A) \stackrel{F\eta_A}\lra FF(A)\stackrel{\mu_A}\lra F(A) \stackrel{\varrho}\lra A .$$  
% (= F(A)\stackrel{\vartheta_A}\lra  F(A) \stackrel{\varrho}\lra A). 
 
In particular, the natural transformation $\mu:FF\to F$ is  compatible if
$$FF \stackrel{\mu}\lra F =  FF\stackrel{F\eta F}\lra FFF\stackrel{\mu F}\lra  FF
\stackrel{\mu}\lra F.$$
It is easy to see that this implies
$$ FF\stackrel{\vartheta \ul{\vartheta} }\lra FF \stackrel{\mu}\lra F = FF \stackrel{\mu}\lra F.$$

Let $\rA_F$ denote the full subcategory  of $\uA_F$ made up by the compatible $F$-modules.  
If $\mu$ is compatible, the image of the free functor $\phi_F$ lies in $\rA_F$ and 
(by restriction or corestriction)  we get the functor pair (keeping the notation for the functors)
 $$ \phi_{F}: \A\to \rA_F, \quad {U}_F: \rA_F \to \A,$$
and a  pairing $(\phi_F,U_F,\alpha_F,\beta_F)$ between $\A$ and $\rA_F$.

Since for $(A,\varrho)$ in $\uA_F$, $\beta_F(I_{U_F(A)}) = \varrho$, the compatibility
condition on $\varrho$ implies that $ \beta \cdot  \alpha \cdot  \beta(\varrho) = \beta(\varrho)$, i.e.,
$\ul\beta$ is regular in $( \phi_F, {U}_F, \alpha_F, \beta_F)$ when restricted to $\rA_F$.
\etm

\btm \label{def-weak}{\bf Definition.} \em
A $q$-unital monad $(F,\eta,\mu)$ is called 
\smallskip

\begin{tabular}{rcl} 
{\em $r$-unital} &if& $\eta$ is regular and $\mu$ is compatible; \\
{\em weak monad} &if&  $(F,\eta,\mu)$ is $r$-unital and $\eta$ is symmetric.
\end{tabular} 
\etm

Summarising the observations from \ref{mon-cont} we have:

\btm \label{r-mon-cont}{\bf Proposition.} Let $(F,\mu,\eta)$
be a $q$-unital monad. 
\begin{zlist}
\item The following are equivalent:
\begin{blist}
\item $(F,\mu,\eta)$ is an $r$-unital monad;  
\item 
$(\phi_F,{U}_F,\alpha_F,\beta_F)$ is a regular pairing of functors between $\A$ and $\rA_F$.
\end{blist}

\item The following are equivalent:
\begin{blist}
\item $(F,\mu,\eta)$ is weak monad;  
\item 
$(\phi_F,{U}_F,\alpha_F,\beta_F)$ is a regular pairing between $\A$ and $\rA_F$ with $\alpha_F$ symmetric.
\end{blist}
\end{zlist}
\etm
\medskip

A quasi-unit $\eta$ that is regular and symmetric is named {\em pre-unit} in the 
literature (e.g.  \cite[Definition 2.3]{FerGon});
for the notion of a weak monad (also called {\em demimonad}) see e.g. \cite{B-weak}, \cite{BoLaSt}. 
 In case $\eta$ is a unit, $q$-unital monads, $r$-unital monads and weak monads all are 
(unital) monads. In (non-unital) algebras over commutative rings, $r$-unital monads are 
obtained from idempotents while weak monads correspond to central idempotents (see \ref{quasi-alg}).

\btm \label{F-mod}{\bf Properties of weak monads.}  
Let $(F,\mu,\eta)$ be a weak monad. 
\begin{rlist}
\item $\vartheta: F\to F$ is a morphism of $q$-unital monads;  
\item for any $(A,\varphi)\in \rA_F$,
$$ 
F(A)\stackrel{\varphi}\lra A = F(A)\stackrel{\varphi}\lra A  
 \stackrel{\eta_A}\lra F(A) \stackrel{\varphi}\lra A 
$$ 
and \; $A\stackrel{\eta_A}\lra F(A) \stackrel{\varphi}\lra A$ \;
is an idempotent $F$-morphism.  
\end{rlist} 
\etm
 
In a $q$-unital monad $(F,\mu,\eta)$, if $\eta$ is regular, a compatible multiplication for $F$ can 
be found. More precisely one can easily show:

 \btm\label{F-m-e-regular}{\bf Proposition.} 
Let $(F,\mu,\eta)$ be a $q$-unital monad.
\begin{zlist}
 \item If $\eta$ is regular, then, for 
$\widetilde\mu:=  \mu \cdot F\mu \cdot  \mu  F\eta F:FF\to F$,
% $$\tilde\mu:FF\stackrel{F\eta F}\lra FFF \stackrel{F\mu}\lra FF\stackrel{\mu}\lra F,$$ 
 $(F,\widetilde\mu,\eta)$ is an $r$-unital monad. 
\item If $\mu$ is compatible, then, for 
 $\widetilde\eta:= \mu  \cdot  F\eta \cdot \eta: I_A \to F$, 
%$$\widetilde\eta:\id_\A \stackrel{\eta}\lra F\stackrel{F\eta}\lra FF \stackrel{\mu}\lra F,$$
 $(F,\mu,\widetilde\eta)$ is an $r$-unital monad.  
\item If $(F,\mu,\eta)$ is an $r$-unital monad, then for 
$$\widehat\mu:
FF\stackrel{\eta FF\eta}\lra FFFF \stackrel{\mu FF}\lra FFF\stackrel{\mu F}\lra FF
 \stackrel{\mu}\lra F,$$
 $(F,\widehat\mu,\eta)$ is a weak monad.
\end{zlist}
\etm

As a special case, we consider $q$-unital monads on the category $_R\M$ of
modules over a  commutative ring $R$ with unit. In the terminology used here 
this comes out as follows. 

\btm\label{quasi-alg}{\bf Non-unital algebras.} \em
 A {\em $q$-unital $R$-algebra} $(A,m,u)$ is a non-unital $R$-algebra $(A,m)$
with some $R$-linear map $u:R\to A$.  
Put $e:=u(1_R)\in A$. Then:
\begin{zlist}
\item $u$ is regular if and only if $e$ is an idempotent in $A$.
\item  $u$ is regular and symmetric if and only if $e$ is a central 
 idempotent (then $Ae$ is a unital $R$-subalgebra of $A$).  
\item $\mu$ is compatible if and only if $ab=aeb$ for all $a,b\in A$. 

\item If $u$ is regular, then $\widetilde m(a\ot b):=aeb$, for $a,b\in A$, defines an 
  $r$-unital algebra $(A,\widetilde m,u)$ ($\widetilde m$ and $u$ are regular).
\item If $u$ is regular, then $\widehat m(a\ot b):=eaebe$, for $a,b\in A$, defines an 
  $r$-unital algebra $(A,\widehat m,u)$ with $u$ symmetric.
\end{zlist}
\etm

Clearly, the $q$-unital algebras $(A,m,u)$ over $R$ correspond to the 
$q$-unital monads given by $(A\ot_R-, m\ot-, u\ot-)$ on $_R\M$.

 For an $A$-module 
$\varrho: A\ot M \to M$, writing as usual $\varrho(a\ot m)= am$, 
 the compatibiliy condition comes out as 
$am=aem$ for all $a\in A$, $m\in M$. 

%\btm\label{quasi-module}{\bf Compatible $A$-modules.} \em
%Let $(A,m,u)$ be an $r$-unital algebra over $R$.
%For the category $_A\rM$ of compatible  $A$-modules, the free functor 
%  $$\phi_A : {_R\M} \to {_A\rM},\quad X\mapsto (A\ot_R X, m_A\ot \id_X),$$ 
%together with the forgetful functor $U_A: {_A\rM} \to {_R\M}$ yield
%a  regular pairing $(\phi_A, U_A, \alpha_A,\beta_A)$ 
%with the maps, for $X\in {_R\M}$, $(M,\rho)\in {_A\rM}$,
%$$\begin{array}{ll}
%\alpha_{A}: \Mor_\A(A\ot_RX,M)\to \Mor_R(X,M), & 
% f\mapsto f\dcirc (u\ot A),\\[+1mm]
% \beta_{A}:\Mor_R(X,M)\to\Mor_\A(A\ot_R X,M),& g\mapsto \rho  \dcirc (A\ot g).
%\end{array}$$ \etm

 \begin{thm} \label{q-mon-act} {\bf Monads acting on functors.} \em
Let $T:\A\to \B$ be a  functor and $(G,\mu',\eta')$ a $q$-unital monad on $\B$. 
We call $T$ a left {\em  $G$-module} if there exists a natural transformation
$\varrho: GT\to T$ such that 
$$GGT\stackrel{G\varrho}\lra GT \stackrel{\varrho}\lra T = 
 GGT\stackrel{\mu'T}\lra GT \stackrel{\varrho}\lra T,$$
and we call it a  {\em compatible $G$-module} if in addition
$$GT\stackrel{\varrho}\lra T =
GT\stackrel{G\eta'}\lra GGT \stackrel{\mu' T}\lra GT \stackrel{\varrho}\lra T.$$
\end{thm}
 
  \btm \label{pop-mon} {\bf Proposition.} 
Let $T:\A\to \B$ be a functor and $(G,\mu',\eta')$ a weak monad on $\B$.
Then the following are equivalent:
\begin{blist} 
\item there is a functor $\overline T: \A\to \rB_G$ with 
$T=U_G \overline T$;
\item  $T$ is a compatible $G$-module.
\end{blist}
\etm

\begin{proof}
(b)$\Ra$(a) Given $T$ as a compatible $G$-module with  
$\varrho: GT\to T$, a functor with the required properties is   
$$\overline T: \A\to \rB_G, \quad A\;\mapsto \; (T(A), \varrho_A:GT(A)\to T(A)).$$
  
(a)$\Ra$(b) 
For any $A\in \A$, there are morphisms  $\rho_A:GT(A)\to T(A)$ and we claim
that these define a natural 
transformation $\rho:GT\to T$. For this we have to show that, 
for any morphism $f:A\to \widehat{A}$,
the middle rectangle is commutative in the diagram
$$\xymatrix{ GGT(A)\ar[rr]^{\mu'_T(A)} && GT(A) \ar[ddd]^{GT(f)}\ar[dl]_{\rho_A} \\
GT(A) \ar[u]^{G\eta' T_A} \ar[d]_{GT(f)} \ar[r]^{\rho_A} & T(A)\ar[d]^{T(f)} & \\
GT(\widehat{A})\ar[d]_{G\eta' T_{\widehat{A}}} \ar[r]^{\rho_{\widehat{A}}} & T(\widehat{A})  \\
 GGT(\widehat{A})\ar[rr]^{\mu'_{T(\widehat{A})}} && GT(\widehat{A}) \ar[lu]^{\rho_{\widehat{A}}} .} $$
The top and bottom diagrams are commutative by compatibility of the $G$-modules, 
 the right trapezium is commutative since $T(f)$ is a $G$-morphism,
and the outer paths commute by symmetry of $\eta'$. 
Thus the inner diagram is commutative showing naturality of $\rho$. 
\end{proof}

\section{Comonads and comodules}\label{qu-comon}

In this section we sketch the transfer of the constructions for monads to comonads.
 
\btm \label{quasi-comon} {\bf $q$-counital comonads and their comodules.} \em
  A {\em functor with coproduct} (or {\em non-counital comonad}) is a pair $(G,\delta)$
where  $G:\A\to \A$ is an endofunctor and  $\delta:G\to GG$ is a natural transformation 
subject to the coassociativity condition $G\delta\cdot \delta=\delta G\cdot \delta$.
 
 For $(G,\delta)$, a {\em (non-counital) $G$-comodule} is defined as an object 
$A\in \A$ with a morphism $\up: A\to G(A)$  in $\A$ such that 
$G\up \dcirc \up = \delta_A \dcirc \up$.

{\em Morphisms} between  $G$-comodules $(A,\up)$, 
$(A',\up')$ are morphisms $g:A\to A'$ in $\A$ satis\-fying
$\up' \dcirc g = G(g)\dcirc \up$, 
and the set of all these is denoted by $\Mor^G(A,A')$.
With these morphisms, (non-counital) $G$-comodules form a category which we 
denote by $\uA^G$. For this there are the obvious free and forgetful functors
 $$\phi^G:\A\to \uA^G,  \quad U^G:\uA^G\to \A. $$

A triple $(G,\delta,\ve)$ is said to be a {\em $q$-counital comonad}  provided
$(G,\delta)$ is a functor with coproduct and $\ve:G\to \id_\A$ is any natural transformation,  
called a {\em quasi-counit}.
One can always define natural transformations
$$\gamma: G\stackrel{\delta}\lra GG \stackrel{G\ve}\lra G,\quad 
\ul{\gamma}: G\stackrel{\delta}\lra GG \stackrel{\ve G}\lra G.$$
 
  Morphisms of  $q$-counital comonads are defined in an obvious way (dual to 
\ref{quasi-mon}). 
\etm

\btm \label{q-co-com} {\bf $q$-counital comonads and  pairings.} \em
 For  $(G,\delta,\ve)$, the functors  $\phi^G$ and $U^G$ allow for a pairing 
$(U^G,\phi^G ,\alpha^G,\beta^G)$
where, for $A\in \A$ and $(B,\up)\in \uA^G$,
$$\begin{array}{ll}
 \alpha^G: \Mor_\A(U^G(B),A)\to \Mor^G (B,\phi^G(A)),& 
f \; \mapsto \; G(f)\dcirc \up, \\[+1mm]  
\beta^G: \Mor^G (B,\phi^G(A)) \to \Mor_\A(U^G(B),A), & 
 g \; \mapsto \; \ve_A \dcirc g.
\end{array}$$   

The quasi-counit $\ve$ is called  {\em regular} if $\beta^G$ is regular, that is,   
 $$G \stackrel{\ve}\lra \id_\A = G \stackrel{\delta}\lra GG 
  \stackrel{G\ve}\lra G \stackrel{\ve}\lra \id_\A,$$
and we say $\eta$ is {\em symmetric} provided $\phi^G$ is so, that is $\gamma=\ul\gamma$.
\smallskip 

A (non-counital) $G$-comodule $(B,\up)$ is said to be 
{\em compatible} provided $\alpha^G\beta^G(\up)=\up$, that is
$$B\stackrel{\up}\lra G(B)=B\stackrel{\up}\lra G(B)\stackrel{\delta_B}\lra GG(B) \stackrel{G\ve_B}\lra G(B).$$ 
In particular, $\delta$ is compatible if
$$G \stackrel{\delta}\lra GG = 
G \stackrel{\delta}\lra GG \stackrel{\delta G}\lra GGG \stackrel{G\ve G}\lra GG.$$
This obviously implies 
$$ G\stackrel{\delta}\lra GG = G\stackrel{\delta}\lra GG \stackrel{\gamma \ul\gamma }\lra GG.$$
By $\rA^G$  we denote the full subcategory
 of $\uA^G$ whose objects are compatible $G$-comodules. 

If $\delta$ is compatible, the image of the free functor $\phi^G$ lies in $\rA^G$ and 
(by restriction and corestriction) we obtain the functor pairing (keeping the notation for the functors)
$$\phi^G: \A \to \rA^G, \quad  U^G: \rA^G \to \A,$$
leading to a pairing  
$( U^G,\phi^G ,\alpha^G,\beta^G)$ between $\A$ and $ \rA^G$.

Since for $(B,\up)$ in $\uA^G$, $\alpha^G(I_{U^G(B)}) = \up$, the compatibility
condition on $\up$ implies that $\alpha^G \cdot \beta^G \cdot \alpha^G (\up)= \alpha^G(\up)$,
 i.e., $\alpha$ is regular in $( U^G,\phi^G ,\alpha^G,\beta^G)$ when restricted to $\rA^G$.
\etm

\btm \label{def-weak-co}{\bf Definition.} \em
A $q$-counital comonad $(G,\delta,\ve)$ is called 
\smallskip

\begin{tabular}{rcl} 
{\em $r$-counital} &if& $\ve$ is regular and $\delta$ is compatible; \\
{\em weak comonad} &if&  it is $r$-counital and $\ve$ is symmetric.
\end{tabular} 
\etm 

From the constructions above we obtain:

\btm \label{r-comon-cont}{\bf Proposition.} Let $(G,\delta,\ve)$ be a $q$-counital comonad. 
\begin{zlist}
\item The following are equivalent:
\begin{blist}
\item $(G,\delta,\ve)$ is an $r$-counital comonad;  
\item $(U^G,\phi^G ,\alpha^G,\beta^G)$ is a regular pairing of functors 
between $\A$ and $\rA^G$.
\end{blist}

\item The following are equivalent:
\begin{blist}
\item  $(G,\delta,\ve)$ is weak comonad;  
\item 
$( U^G,\phi^G ,\alpha^G,\beta^G)$ is a regular pairing of functors between $\A$ 
and $\rA^G$ with $\beta^G$ symmetric.
\end{blist}
\end{zlist}
\etm 

Similar to the situation for modules, for any (counital) comonad $(G,\delta,\ve)$, 
all non-counital $G$-comodules are compatible (i.e., $\uA^G=\rA^G$). 

\btm \label{G-comod}{\bf Properties  of weak comonads.}  
 Let $(G,\delta,\ve)$ be a weak comonad.
\begin{rlist} 
\item  $\gamma:G\to G$ is an idempotent morphism of $q$-counital comonads; 
 
\item  for any $(B,\up)\in \rA^G$, 
$$ B\stackrel{\up}\lra G(B) = B\stackrel{\up}\lra G(B)\stackrel{\ve_B}\lra B
\stackrel{\up}\lra G(B) $$  
and $B\stackrel{\up}\lra  G(B) \stackrel{\ve_B}\lra B$ 
is an idempotent $G$-morphism. 
\end{rlist}
\etm

Properties of  pairings can improved in the following sense. 

\btm\label{G-e-regular}{\bf Proposition.} 
Let $(G,\delta,\ve)$ be a $q$-counital comonad.
% with related pairing $(U^G,\phi^G,\alpha^G,\beta^G)$.

\begin{zlist}
 \item If $\ve$ is regular, then, for 
$\widetilde\delta:G\stackrel{\delta}\lra GG \stackrel{G\delta}\lra GGG
 \stackrel{G\ve G}\lra GG$, 
 $(G,\widetilde\delta,\ve)$ is an $r$-counital comonad. 

\item If $\delta$ is compatible, then, for 
$\widetilde\ve:G\stackrel{\delta}\lra GG\stackrel{G\ve}\lra G 
 \stackrel{\ve}\lra \id_\A$, 
 $(G,\delta,\widetilde\ve)$ is an $r$-counital comonad. 

\item If $(G,\delta,\ve)$ is a regular quasi-comonad, then, for  
$$\widehat\delta:G\stackrel{\delta}\lra GG \stackrel{G\delta}\lra GGG
 \stackrel{GG\delta}\lra GGGG
 \stackrel{\ve GG \ve}\lra GG,$$
$(G,\widehat\delta,\ve)$ is a weak comonad. 
\end{zlist}
\etm

As a special case, consider  non-counital comonads on the category $_R\M$ of
modules over a commutative ring $R$ with unit. In our terminology  
this comes out as follows.

\btm\label{quasi-coalg}{\bf Non-counital coalgebras.} \em
 A {\em $q$-counital coalgebra} $(C,\Delta,\ve)$ is 
a non-counital $R$-coalgebra $(C,\Delta)$ with some $R$-linear map  $\ve:C\to R$.
Writing  $\Delta(c)= \sum c_{\underline 1} \ot c_{\underline 2}$   for $c\in C$, 
  we have:

\begin{zlist}
\item $\ve$ is regular if and only if for any 
  $c\in C$, $\ve(c)= \sum \ve(c_{\underline 1}) \ve(c_{\underline 2})$.
\item  $\ve$ is symmetric if and only if 
  $\sum c_{\underline 1} \ve(c_{\underline 2})=\sum \ve(c_{\underline 1})
 c_{\underline 2}.$

\item $\Delta$ is compatible if and only if 
  $\Delta(c)= \sum c_{\underline 1} \ot  c_{\underline 2}\ve( c_{\underline 3})$.
% \item $\Delta$ is symmetric if and only if
%  $\sum c \ot \ve(d_{\underline 1})d_{\underline 2}  =
% \sum c_{\underline 1} \ve(c_{\underline 2})\ot d .$

\item If $\ve$ is regular, then 
$\widetilde\Delta(c):=\sum c_{\underline 1}\ot \ve(c_{\underline 2})c_{\underline 3}$  
  defines an $r$-counital coalgebra $(C,\widetilde\Delta,\ve)$. 

\item If $(C,\Delta,\ve)$ is an $r$-counital comonad, then 
$\widehat\Delta(c):=\sum \ve(c_{\underline 1})c_{\underline 2} 
\ot c_{\underline 3} \ve (c_{\underline 4}) $
defines an $r$-counital coalgebra $(C,\widehat\Delta,\ve)$ with $\ve$ symmetric.
\end{zlist}
\etm

Clearly, the $q$-counital coalgebras $(C,\Delta,\ve)$ over $R$ correspond to the 
$q$-counital comonads given by $(C\ot_R-, \Delta\ot-, \ve\ot-)$ on $_R\M$.
From this the compatibility conditions for $C$-comodules are derived (see \ref{q-co-com}).

\btm\label{w-coring}{\bf Weak corings and pre-$A$-corings.} \em 
Let $A$ be a ring with unit $1_A$
and $\cC$ a non-unital $(A,A)$-bimodule which is unital as right $A$-module. 
Assume there are $(A,A)$-bilinear maps 
$$ \uDelta: \cC \to \cC\ot_A \cC, \quad \ueps: \cC\to A,$$ 
where $\uDelta$ is coassociative.

$(\cC,\uDelta,\ueps)$ is called a {\em right unital weak $A$-coring} 
in \cite{Wi-weak}, provided for all $c\in \cC$,
$$(\ueps\ot \id_\cC)\dcirc \uDelta(c) = 1_A\cdot c = (\id_\cC\ot \ueps )\dcirc \uDelta(c),$$ 
which reads in (obvious) Sweedler notation as \;
 $  \sum \ueps(c\se)c\sz =1_A\cdot c = \sum c\se\ueps(c\sz).$ 

From the equations
$$\begin{array}{rl} 
(\id_\cC\ot \ueps\ot \id_\cC) \dcirc(\id_\cC\ot \uDelta) \dcirc\uDelta(c)& = \;\;
 \sum c\se \ot 1_A \cdot c\sz\; =\; \sum c\se \ot c\sz\; = \;\uDelta(c), \\[+1mm]
 (\id_\cC\ot \ueps\ot \id_\cC)   \dcirc(\uDelta\ot\id_\cC) \dcirc\uDelta  (c) & = \;\;
  \sum 1_A\cdot c\se \ot   c\sz\; =\; 1_A\cdot  \uDelta (c),
\end{array}
$$
it follows by coassociativity that $1_A\cdot \uDelta (c)= \uDelta (c)$.
Summarising we see that,
in this case, $(\cC,\uDelta,\ueps)$ induces a    
weak comonad on the category $_A\uM$ of left 
non-unital $A$-modules  (=$_A\rM$ since $A$ has a unit). 
\smallskip
 
$(\cC,\uDelta,\ueps)$ is called an {\em $A$-pre-coring} in 
\cite[Section 6]{Brz:str}, if 
$$ (\ueps\ot \id_\cC)\dcirc \uDelta(c) = c, \quad 
   (\id_\cC\ot  \ueps)\dcirc \uDelta(c) = 1_A\cdot c,$$
which reads (in Sweedler notation) as
\; $c = \sum \ueps(c\se)c\sz$, \; $1_A\cdot c =\sum c\se\ueps(c\sz).$ 

Similar to the computation above we obtain that $1_A\cdot \uDelta (c)= \uDelta (c)$.
Now $(\cC,\uDelta,\ueps)$ induces an $r$-counital comonad on $_A\uM$ 
but   $\ueps$ is not symmetric. 

Notice that in both cases considered above, restriction and corestriction 
of $\uDelta$ and $\ueps$ yield an {\em $A$-coring}  $(A\cC,\uDelta,\ueps)$ 
(e.g. \cite[Proposition 1.3]{Wi-weak}).
\etm

\begin{thm} \label{q-comon-act} {\bf Comonads acting on functors.} \em
Let $T:\A\to \B$ be a  functor and $(G,\delta,\ve)$ a weak comonad on $\B$. 
We call $T$ a left {\em (non-counital) $G$-comodule} if there exists a natural transformation
$\up: T\to GT$ such that 
$$T\stackrel{\up}\lra GT \stackrel{\up G}\lra GGT = 
 T\stackrel{\up T}\lra GT \stackrel{\delta}\lra GGT,$$
and we call it a {\em compatible $G$-comodule} if, in addition,
$$T\stackrel{\up}\lra GT =
T\stackrel{\up}\lra GT \stackrel{\delta}\lra GGT \stackrel{G\ve}\lra GT.$$

Dual to Proposition \ref{pop-mon}, given a weak comonad $(G,\delta,\ve)$ on $\B$,
a functor $T:\A\to \B$ is a compatible $G$-comodule if and only if   
there is a functor $\overline T: \A\to \rB^G$ with $T=U^G  \overline T$.
\end{thm}

The motivation for considering generalised monads and comonads came from structures 
observed while handling full pairings of functors (see end of Section \ref{adj-con}).
Now we want to reconsider the pairings in view of these constructions.   

For any pairing $(L,R,\alpha,\beta)$ between the categories $\A$ and $\B$,   
$(RL,R\ve L, \eta)$ is a $q$-unital monad and 
 $(LR, L\eta R, \ve)$ is a $q$-counital comonad. 
It is easy to see that 
\begin{rlist}
\item {\em if $\beta$ is regular, then for any $B \in \B$,  $R\ve: RLR(B)\to R(B)$ is a compatible
    $RL$-module.}
\item {\em if $\alpha$ is regular, then for any $A\in \A$,  $L(A), L\eta: L(A)\to LRL(A)$ is a compatible 
       $LR$-comodules.}
\end{rlist}

\btm \label{cont-comon}{\bf Comparison functors.}  
  For a regular pairing $(L,R,\alpha,\beta)$ between  $\A$ and $\B$,    
 \smallskip

  $(RL, R\ve L, \eta)$ is an $r$-unital monad on $\A$  with a (comparison) functor 
 $$\widehat R : \B \to \rA_{RL}, \quad B\mapsto (R(B), R\ve: RLR(B)\to R(B)),$$

  $(LR, L\eta R, \ve)$ is an $r$-counital comonad on $\B$  with a (comparison) functor 
 $$\widetilde L: \A \to \uB^{LR}, \quad A\mapsto (L(A), L\eta: L(A)\to LRL(A)),$$

inducing commutativity of the diagrams
$$\xymatrix{\A\ar[r]^L \ar[dr]_{\phi_{RL}} & \B \ar[d]^{\wR } \ar[r]^R & \A \\
  & \rA_{RL} \ar[ru]_{U_{RL}}& ,  } \qquad 
\xymatrix{\B\ar[r]^R \ar[dr]_{\phi^{LR}} & 
 \A \ar[d]^{\widetilde L} \ar[r]^L & \B \\
  & \rB^{LR} \ar[ru]_{U^{LR}} & . } 
$$
 \etm
 
It follows from  \ref{mon-cont} that for the  $r$-unital monad
 $(RL, R\ve L, \eta)$, 
 $(\phi_{RL},U_{RL},  \alpha_{RL},  \beta_{RL})$ is a regular pairing   between 
$\A$ and $\rA_{RL}$. Similarly,  by  \ref{q-co-com}, for the $R$-countial comonad
 $(LR, L\eta R, \ve)$, 
  $(U^{LR}, \phi^{LR},  \alpha^{LR},  \beta^{LR})$ is a regular pairing  between 
 $\B$ and $\ul\B^{LR}$.

\btm \label{alpha-symm}{\bf Relating $(L,R)$ with
$(\phi_{RL},U_{RL})$ and $(U^{LR}, \phi^{LR})$.} \em
With the above notions we form the diagram
  $$\xymatrix{ \Mor_\B(L(A),B) \ar[d]_{\wR_{-,-}} \ar[r]^{\alpha} &  
  \Mor_\A(A,R(B)) \ar[r]^{\beta } &  \Mor_\B(L(A),B) 
   \ar[d]^{\wR_{-,-}}\\
   \Mor_{RL}(\phi_{RL}(A),R(B)) \ar[r]^{\alpha_{RL}} &  
   \Mor_\A(A,U_{RL}R(B)) \ar[r]^{\beta_{RL}} &   \Mor_{RL}(\phi_{RL}(A),R(B)).}$$

 This diagram is commutative if and only if $\alpha$ is symmetric  
(see Definitions \ref{a-regular}).
\smallskip

Similar constructions apply for $(L,R)$, $(U^{LR}, \phi^{LR})$ and $\tL_{-,-}$.
and $\beta $  is symmetric if and only if  
$\tK_{-,-}\dcirc \alpha \dcirc \beta = \alpha^{LR}\dcirc \beta^{LR} \dcirc \tL_{-,-}$. 
\etm

 \btm \label{reg-full}{\bf Corollary.} Consider  a pairing $(L,R,\alpha,\beta)$ (see \ref{setting}).
\begin{zlist} 
\item The following are equivalent:
\begin{blist}
\item $(L,R,\alpha,\beta)$ is a regular pairing;
\item  $(RL, R\ve L, \eta)$ is an $r$-unital monad on $\A$ and \\
 $(LR, L\eta R, \ve)$ is an $r$-counital comonad on $\B$.
\end{blist}
 \item The following are equivalent:
\begin{blist}
\item $(L,R,\alpha,\beta)$ is a regular pairing with $\alpha$ and $\beta$ symmetric;
\item  $(RL, R\ve L, \eta)$ is a weak monad on $\A$ and \\
 $(LR, L\eta R, \ve)$ is a weak comonad on $\B$.
\end{blist}
\end{zlist}
\etm

\section{Entwining monads and comonads}\label{entw-qu-mon}
  
 \begin{thm} \label{lift-f}  {\bf Lifting of functors to module categories.} \em 
 Let $(F,\mu,\eta)$ and $(L,\mu',\eta')$ be $r$-unital monads on the categories $\A$ and $\B$, 
respectively,  and $\rA_F$, $\rB_L$ the categories of the corresponding compatible modules 
 (see \ref{mon-cont}). Given functors 
 $T: \A\to \B$\,   and  \,  $\oT: \rA_F\to \rB_L$,
  we say that  $\oT$ {\em is a lifting of $T$} provided the diagram
\begin{equation}\label{lift-dia}
\xymatrix{ 
    \rA_F \ar[r]^\oT \ar[d]_{U_F} & \rB_L \ar[d]^{U_L} \\
  \A \ar[r]^T &\B }
\end{equation}
is commutative, where the $U$'s denote the forgetful functors. 
 \end{thm} 
 
\begin{thm} \label{nat.trans.p} {\bf Proposition.} 
With the data given in {\rm \ref{lift-f}}, 
consider the functors $TF,\, LT: \A\to \B$ and
 a natural transformation $\lambda: LT\to TF$.
The non-unital $F$-module $(F,\mu)$ induces an $L$-action on $TF$,
  $$\chi: LTF\stackrel{\lambda F}\lra
TFF\stackrel{T\mu}\lra TF.$$
\begin{zlist}
\item If $(TF,\chi)$ is a (non-unital) $L$-module, then we get the commutative diagram  
\begin{equation}\label{lift-equ}
\xymatrix{ 
L  L  T\ar[r]^{L  \lambda} \ar[d]_{\mu'  T} &
   L  T F \ar[r]^{L T\vartheta  } & L  T F \ar[r]^{\lambda  F} & 
   T F F  \ar[d]^{T  \mu} \\
  L  T \ar[rr]^{\lambda}& & T F \ar[r]^{T\vartheta}  & T F .} 
\end{equation}

\item If   $(TF,\chi)$ is a compatible $L$-module, then  
(with $\vartheta' =\mu'\cdot F\eta'$)
\begin{equation}\label{lift-equ-reg}
 \xymatrix{LT \ar[r]^{\vartheta' T} &
LT  \ar[r]^{\lambda} & TF \ar[r]^{T\vartheta }  & TF} =
\xymatrix{LT \ar[r]^{\lambda}  & TF \ar[r]^{ T\vartheta}  & TF. }
\end{equation}

\item If $\eta$ is symmetric in $(F, \mu,\eta)$ and $(A,\varphi)$ is a compatible $F$-module, then 
\begin{equation}\label{F-reg-2}
T\varphi \dcirc \lambda_A = 
   T\varphi\dcirc\lambda_A \dcirc LT\varphi \dcirc LT\eta_A .
\end{equation}   
\end{zlist}
\end{thm}
\begin{proof} The proof follows essentially as in the monad case replacing 
the identity on $F$ at some places by  $\vartheta=\mu\dcirc F\eta$ (see \ref{quasi-mon}).

To show (3), Proposition \ref{F-mod} is needed.
\end{proof}

 \begin{thm} \label{reg-lift-p} {\bf Proposition.} 
  Let $(F,\mu,\eta)$ and $(L,\mu',\eta')$ be $r$-unital monads on $\A$ and $\B$, respectively, and $T:\A\to \B$ any functor.
Then a natural transformation $\lambda: L T\to T F$ induces a lifting 
to the compatible modules,
$$\oT:\rA_F \to \rB_L, \quad (A,\varphi) \mapsto 
  (T(A), T\varphi \dcirc \lambda_A:LT(A)\to T(A) ),$$  
if and only if the diagram {\em (\ref{lift-equ})} is commutative 
and  equation {\em (\ref{lift-equ-reg})} holds.
\end{thm}
\begin{proof} One direction follows from 
Proposition \ref{nat.trans.p}, the other one by a slight modification of the
proof in the monad case.
\end{proof}

To show that the lifting property implies the existence of a natural transformation 
$\lambda: L T\to T F$ we need the symmetry of the units, that is, we require 
the $r$-unital monads to be weak monads.  
Then we can extend Applegate's lifting theorem for monads (and unital modules) 
(e.g. \cite[Lemma 1]{John}, \cite[3.3]{W-CoAl}) to weak monads (and compatible modules).  

\begin{thm} \label{reg-lift} {\bf Theorem.} 
  Let $(F,\mu,\eta)$ and $(L,\mu',\eta')$ be weak monads on 
 $\A$ and $\B$, respectively.  
For any functor $T:\A\to \B$, there are bijective correspondences between  
\begin{rlist}
\item liftings of $T$ to $\oT: \rA_F\to \rB_L$;
\item compatible $L$-module structures $\varrho$ on $TU_F:\A_F\to \B$;  
\item  natural transformations $ \lambda: L T\to T F$ 
 with commuting diagrams
\begin{equation}\label{lift-equ-r}
\xymatrix{ 
L  L  T\ar[r]^{L    \lambda} \ar[d]_{\mu' T} &
 L  T F \ar[r]^{  \lambda  F} & 
   T F F  \ar[d]^{T \mu} \\
  L  T \ar[rr]^{ \lambda}& & T F  ,} \qquad 
\xymatrix{ LT \ar[r]^{\vartheta' T} \ar[d]_\lambda \ar[dr]^\lambda  & LT \ar[d]^\lambda \\
           TF \ar[r]_{T\vartheta} & TF .}
\end{equation}
\end{rlist}
\end{thm}
 \begin{proof}
(i)$\LRa$(ii) follows by  Proposition \ref{pop-mon}. 
\smallskip

(ii)$\Ra$(iii) Given the compatible $L$-module structure map  $\varrho$, put  
 $$\lambda:= \varrho F\dcirc LT\eta: LT\stackrel{LT\eta}\lra LTF 
  \stackrel{\varrho F}\lra TF.$$ 
 Notice that for $\lambda$ we can take $T\vartheta\cdot\lambda$ from Proposition \ref{nat.trans.p}.

(iii)$\Ra$(i) Given $\lambda$ with the commutative diagram in (iii),
it follows by Propositions \ref{reg-lift-p} that  
 $\varrho_A:=T\varphi\dcirc \lambda_A$ 
induces a lifting.
\end{proof}

%Recall that the natural transformation $T\vartheta$ (see \ref{lift-f}) shows the deviation of
% the quasi-unit from unitality. We list some properties and relations which are of use for the proofs.
%\begin{thm}\label{Apple-prop}  {\bf Lemma.} 
%Let $(F,\mu,\eta)$,  $(L,\mu',\eta')$ be $q$-unital monads
%and $T: \A \to \B$ a functor  with (any) natural transformation $\lambda:LT\to TF$ and consider
%$$  \widehat\kappa :\xymatrix{ TF\ar[r]^{\eta' TF} &LTF \ar[r]^{\lambda F} & TFF %\ar[r]^{T\mu}&TF. }$$  
%\begin{zlist}
%\item $\widehat\kappa \dcirc T\vartheta  = T\vartheta  \dcirc \widehat\kappa  $.
%\item
%  If $\lambda \dcirc \eta' T = T \eta$, then $\widehat\kappa= T\underline\vartheta $.
%\item If the diagram {\rm (\ref{lift-equ-r})} is commutative, then 
%$\lambda \dcirc \ul{\vartheta}'T= \widehat\kappa \dcirc \lambda$. 
%
%\item If {\rm (\ref{lift-equ-r})} is commutative and $\eta'$ is regular,
%then $\widehat\kappa$ is idempotent.\end{zlist}\end{thm}
% \begin{proof} All assertions are shown by standard computations.\end{proof}

 \begin{thm} \label{lift-f-co} {\bf Lifting of functors to  comodules.} \em 
 Let $(G,\delta,\ve)$ and $(H,\delta',\ve')$ be $r$-unital comonads on 
the categories $\A$ and $\B$, respectively, 
and  $\rA^G$, $\rB^H$ the corresponding categories of the compatible comodules 
 (see \ref{q-co-com}). Given a functor  
 $T: \A\to \B$, a functor $\wT: \rA^G\to \rB^H$, 
 is said to be {\em is a lifting of $T$} if the diagram
\begin{equation}\label{lift-comod}
\xymatrix{ 
    \rA^G \ar[r]^\wT \ar[d]_{U^G} & \rB^H \ar[d]^{U^H} \\
  \A \ar[r]^T &\B }
\end{equation}  %$U^H \wT = T U^G$,
is commutative 
 where the $U$'s denote the forgetful functors. 
 \end{thm}  

\begin{thm} \label{nat.trans.p.co} {\bf Proposition.} 
With the data given in {\rm \ref{lift-f-co}},  consider the  functors $TG,\, HT: \A\to \B$ and
 a natural transformation $\psi: TG\to HT$.
The (non-counital) $G$-comodule $(G,\delta)$ induces an $H$-coaction on $TG$,
  $$\zeta: TG\stackrel{T\delta}\lra TGG\stackrel{\psi G}\lra HTG.$$
\begin{zlist}
\item If $(TG,\zeta)$ is a (non-counital) $H$-comodule,  we get the commutative diagram  
\begin{equation}\label{lift-equ-co}
\xymatrix{ 
 TG \ar[r]^{ T\gamma } \ar[d]_{T\delta} &
  T G \ar[rr]^{\psi} & & H  T \ar[d]^{\delta'T}  \\
 TGG \ar[r]^{\psi G}& HTG \ar[r]^{HT\gamma } & HT G \ar[r]^{H\psi} & HHT .} 
\end{equation}

\item If $H$  $(TG,\zeta)$ is a compatible $H$-module, then 
\begin{equation}\label{lift-equ-reg-co}
\xymatrix{TG \ar[r]^{ T\gamma}
 & TG  \ar[r]^{\psi} & HT \ar[r]^{\gamma'T}  & HT} =
\xymatrix{TG  \ar[r]^{T\gamma }& TG\ar[r]^{\psi} & HT. }
\end{equation}

\item  If $\ve$ is symmetric and $(A,\up)$ is a compatible $G$-comodule, then
$$\psi\cdot T\up = HT\ve \cdot HT\up \cdot \psi \cdot T\up.$$
\end{zlist}
\end{thm}
\begin{proof}
The situation is dual to that of Proposition \ref{nat.trans.p}. 
\end{proof}

\begin{thm} \label{reg-lift-p-co} {\bf Proposition.} 
  Let $(G,\delta,\ve)$ and $(H,\delta',\ve')$ be $r$-counital comonads on 
the categories $\A$ and $\B$, respectively, and $T:\A\to \B$ any functor.
A natural transformation $\psi:  TG\to HT$ induces a lifting 
$$\wT:\rA^G \to \rB^H, \quad (A,\up) \mapsto 
  (T(A), \psi \dcirc T\up: T(A)\to HT(A) ),$$ 
if and only if the diagram  {\em (\ref{lift-equ-co})} is commutative and
equation {\em (\ref{lift-equ-reg-co})} holds. 
\end{thm}
\begin{proof} The proof is dual to that of Proposition \ref{reg-lift-p}.
\end{proof}

 Dualising Theorem \ref{reg-lift}, we obtain an extension of Applegate's lifting theorem for  comonads
(and comodules)  (e.g. \cite[3.5]{W-CoAl})
  to weak comonads (and compatible comodules).  

 \begin{thm} \label{reg-lift-co} {\bf Theorem.} 
  Let $(G,\delta,\ve)$ and $(H,\delta',\ve')$ be weak comonads on 
 $\A$ and $\B$, respectively.
%and $\rA^G$ and $\rB^H$ the categories of the compatible comodules.
For any functor $T:\A\to \B$, there are bijective correspondences between  
\begin{rlist}
\item liftings of $T$ to $\wT: \rA^G\to \rB^H$;
\item compatible $H$-comodule structures $\up:TU^G \to HTU^G$;
\item  natural transformations $\psi: TG\to HT$ 
 with commutative diagrams 
$$  
\xymatrix{ 
 TG  \ar[d]_{T\delta}  
 \ar[rr]^{\psi} & & H  T  \ar[d]^{\delta' T}  \\
 TGG \ar[r]^{\psi G}& HTG \ar[r]^{H\psi} & HHT,}  \qquad
\xymatrix{ TG \ar[r]^{T\gamma} \ar[d]_\psi \ar[dr]^\psi & TG \ar[d]^\psi \\
         HT \ar[r]_{\gamma'T} & HT.}
$$  
\end{rlist}
\end{thm}
\begin{proof} In view of \ref{nat.trans.p.co} and \ref{reg-lift-p-co},  
the proof is dual to that of Theorem \ref{reg-lift}. Here we take
$\psi$ as the composition $\psi\dcirc  T\gamma$ (with $\psi$ from \ref{nat.trans.p.co}).
\end{proof}

\section{Lifting of endofunctors to modules and comodules}\label{lift-e-qu-mod}

Given a  weak monad $(F,\mu,\eta)$, or a weak comonad $(G\,\delta,\ve)$, 
and any endofunctor $T$  on the category $\A$, 
we have learned in the preceding sections when $T$ can be lifted to 
an endofunctor of the compatible modules or comodules, respectively. 
Now, one may also ask if the lifting is again a weak monad or a
weak comonad, respectively.  

 \btm \label{qm-entw} {\bf Entwining $r$-unital monads.} For weak monads
 $(F,\mu,\eta)$ and $(T,\cmu,\ceta)$ on $\A$ and a natural transformation
$\lambda: FT\to T F $, the following are equivalent:
\begin{blist}
\item defining product and quasi-unit on $TF$ by
$$ 
 \omu:TFTF\stackrel{T\lambda F}\lra TTFF \stackrel{TT\mu}\lra TTF 
  \stackrel{\cmu F} \lra TF, \quad    
\oeta:\id_\A\stackrel{\eta}\lra F\stackrel{F\ceta}\lra FT
 \stackrel{\lambda}\lra TF , 
$$ 
 yields a weak monad $ (T F ,\omu,\oeta)$ on $\A$;

\item  $\lambda$  induces commutativity of the diagrams  
\begin{equation}\label{lift-equ-r-end}
\xymatrix{ 
F  F  T\ar[r]^{F \lambda} \ar[d]_{\mu T} &
 F  T F \ar[r]^{\lambda  F} & 
   T F F  \ar[d]^{T \mu} \\
  F  T \ar[rr]^{\lambda}& & T F  ,} \qquad
\xymatrix{ FT \ar[r]^{\vartheta T} \ar[d]_\lambda \ar[dr]^\lambda & FT \ar[d]^\lambda \\
         TF \ar[r]_{T\vartheta} & TF,}
\end{equation}

\begin{equation}\label{lift-q-mon}
\xymatrix{ FTT \ar[d]_{F\cmu} \ar[r]^{\lambda T} & TFT \ar[r]^{T\lambda} &
     TTF    \ar[d]^{\cmu F} \\
   FT \ar[rr]^{\lambda}& & TF ,}
\qquad
\xymatrix{ FT \ar[r]^{F\cvartheta } \ar[d]_\lambda \ar[dr]^\lambda & FT \ar[d]^\lambda \\
         TF \ar[r]_{\cvartheta F} & TF;}
\end{equation} 
\item $\lambda$ induces commutativity of the diagrams in  {\rm (\ref{lift-equ-r-end})} 
 and the square in ({\rm \ref{lift-q-mon}}),
 and there are natural transformations 
\begin{center}
 $\cmu  F : TTF\to TF $\quad and \quad $\lambda\dcirc F\ceta:F \to TF $ 
 \end{center}
where $\cmu F$ is a left and right $F$-module morphism and $\lambda\dcirc F\ceta$ is an 
$F$-module morphism.  
\end{blist} 
 If these conditions hold, we obtain morphisms of $q$-uni\-tal monads,
\begin{center}
$\lambda\dcirc F\ceta: F\to TF$ \quad and \quad $\lambda\dcirc \eta T: T\to TF$.
\end{center}  
\end{thm}

\begin{proof} The assertions follow from the general results in Section \ref{entw-qu-mon} and some
routine computations.
\end{proof}

\begin{thm}\label{crossed-2}{\bf Weak crossed products.} \em
Given $(F,\mu,\eta)$ and $T:\A\to \A$, 
the composition $TF$ may have a weak monad structure without 
requiring such a structure on $T$. For example, 
replacing the natural transformations 
 $\cmu  F$ and $\lambda\cdot F\ceta$ in \ref{qm-entw}(c) 
 by some natural transformations  
$$\nu: T T F\to T F, \quad \xi: F\to T F, $$
similar to  \ref{qm-entw}(a), a multiplication and a quasi-unit can be defined on $TF$. 
To make this a weak monad on $\A$, special conditions are to be imposed on $\nu$ and $\xi$
which can be obtained by routine computations.

Having $\nu$ and $\xi$, one also has natural transformations  
$$\bar\nu:\xymatrix{T T \ar[r]^{T T \eta\quad}& T T F 
         \ar[r]^{ \nu\;} & T F ,} \quad 
 \oeta:\xymatrix{\id_\A\ar[r]^{\eta} &F \ar[r]^{\xi\quad} & T F ,}$$
and it is easy to see that
$\bar\nu$ leads to the same product on $TF$ as $\nu$ does.
Thus $\bar\nu$ and $\oeta$ may be used to define a weak monad 
structure on $TF$ and the conditions required come out as 
{\em cocycle} and {\em twisted conditions}. For more details we refer, e.g., to 
 \cite{AFGR}, \cite[Section 3]{FerGon}. 
\end{thm}

For a weak comonad $(G,\delta,\ve)$ and an endofunctor $T:\A\to \A$,
we now consider liftings to the category of  
compatible $G$-comodules, $\wT: \rA^G\to \rA^G$.
The case when $T$ has a weak comonad structure is  dual to \ref{qm-entw}: 
 
\begin{thm} \label{qcom-entw} {\bf Entwining weak comonads. } 
For weak comonads $(F,\delta,\ve)$, $(T,\cdelta,\cve)$, 
and a natural transformation
$\psi: TG\to GT$, the following are equivalent:
\begin{blist}
 \item  defining a coproduct and quasi-counit on $TG$ by 
$$  
 \wdelta: TG\stackrel{\cdelta G}\lra TTG  \stackrel{TT\delta}\lra TTGG
\stackrel{T\psi G}\lra  TGTG, \quad    
\wve: TG \stackrel{\psi}\lra GT \stackrel{G\cve}\lra G \stackrel{\ve}\lra \id_\A ,
$$   
 yields a weak comonad $(TG,\wdelta,\wve)$ on $\A$;

\item  $\psi$ induces commutativity of the diagrams, where $\gamma =  T\ve \dcirc \delta$,
 $\cgamma =  T\cve \dcirc \cdelta $,
\begin{equation}\label{lift-equ-e-co}
\xymatrix{ 
 TG  \ar[d]_{T\delta}  
 \ar[rr]^{\psi} & & G  T  \ar[d]^{\delta T}  \\
 TGG \ar[r]^{\psi G}& GTG \ar[r]^{G\psi} & GGT,} \qquad
\xymatrix{TG \ar[r]^{T\gamma} \ar[d]_\psi \ar[dr]^\psi &  TG \ar[d]^\psi \\
         GT \ar[r]_{\gamma T} & GT,}
\end{equation}
\begin{equation}\label{lift-comon}
\xymatrix{ TG \ar[rr]^\psi \ar[d]_{\cdelta G}& & GT \ar[d]^{G\cdelta} \\
  TTG \ar[r]^{T\psi} & TGT \ar[r]^{\psi T} & GTT ,} \qquad
\xymatrix{TG \ar[r]^{{\cgamma} G} \ar[d]_\psi \ar[dr]^\psi &  TG \ar[d]^\psi \\
         GT \ar[r]_{ G {\cgamma}} & GT,}
\end{equation} 

\item $\psi$ induces commutativity of the diagrams {\rm (\ref{lift-equ-e-co})}
and the square in {\rm (\ref{lift-comon})} and
we have natural transformations 
$$ \cdelta G: TG\to TTG, \quad G\cve \dcirc \psi  : TG\to G,$$ 
where $\cdelta G$ is a left and right  $G$-comodule morphism and $G\cve \dcirc \psi$ is 
a left $G$-comodule morphism. 
\end{blist}
 If these conditions hold, we obtain morphisms of $q$-unital comonads,
\begin{center}  
$G\cve \dcirc \psi: TG\to G$ \, and \, 
$ \ve T  \dcirc \psi: TG\to T$.
 \end{center}                       
\end{thm}

\btm\label{crossed-co}{\bf Weak crossed coproducts.} \em 
In the situation of \ref{qcom-entw},
the coproduct on $TG$ 
can also be expressed by replacing the natural transformations 
$\cdelta G$ and $G\cve\dcirc \psi$ by any natural transformations 
\begin{center}
$ \nu: TG\to TTG $ \; and \; $\zeta:  TG\to G $,
\end{center}
 subject to certain conditions to obtain a weak comonad structure on $TG$. 
 
Given $\nu$ and $\zeta$ as above, one may form 
$$\widehat\nu:\xymatrix{TG\ar[r]^\nu & TTG \ar[r]^{TT\ve} & TT},\quad
   \widehat\zeta: \xymatrix{TG\ar[r]^\zeta & G \ar[r]^{\ve} & \id_\A},$$
and it is easy to see that these induce a weak comonad structure on $TG$.
This leads to the {\em weak crossed coproduct} as considered (for coalgebras)
in \cite{FerGon} and \cite{FeGoRo}, for example.
\etm

\section{Mixed entwinings and liftings}\label{mixed-entw}

Throughout this section let $(F,\mu,\eta)$ denote
 a weak monad and $(G,\delta,\ve)$ a 
weak comonad on any category $\A$.
In this section we investigate the lifting properties to compatible $F$-modules and 
compatible $G$-comodules, respectively. 

\btm\label{lift-func}{\bf Liftings of monads and comonads.} \em 
Consider the diagrams  
$$ \xymatrix{ \rA_F \ar[rr]^\oG \ar[d]_{U_F} & & \rA_F \ar[d]^{U_F} \\
       \A \ar[rr]^{G} & & \A, }  \quad
 \xymatrix{ \rA^G \ar[rr]^\wF \ar[d]_{U^G} & & \rA^G \ar[d]^{U^G} \\
       \A \ar[rr]^{F} & & \A. }$$
In both cases the lifting properties are related to a natural transformation
 $$\omega: FG\to GF.$$ 
The lifting in the left hand case requires commutativity of the diagrams
(Proposition \ref{reg-lift-p})
\begin{equation}\label{m-lift-mon}
\xymatrix{ 
F  F G\ar[r]^{F  \omega} \ar[d]_{\mu G} &
 F  G F \ar[r]^{ \omega  F} & 
   G F F  \ar[d]^{G \mu} \\
  F  G \ar[rr]^{ \omega}& & G F ,}
\quad 
\xymatrix{  F G\ar[r]^{\omega} \ar[d]_{\vartheta G} \ar[dr]^\omega  & GF \ar[d]^{G\vartheta} \\
 F G\ar[r]^{\omega} & GF ,}
\end{equation}
whereas the lifting to $\rA^G$ needs commutativity of the diagrams 
(Proposition \ref{reg-lift-p-co})
\begin{equation}\label{m-lift-com}
\xymatrix{ 
 FG  \ar[d]_{F\delta}  
 \ar[rr]^{\omega} & & G  F  \ar[d]^{\delta F}  \\
 FGG \ar[r]^{\omega G}& GFG \ar[r]^{G\omega} & GGF,} \quad
 \xymatrix{ 
 FG \ar[r]^\omega \ar[d]_{F\gamma} \ar[dr]^\omega  & GF \ar[d]^{\gamma F} \\
 FG \ar[r]^\omega & GF .}
\end{equation}
 
To make $\oG$ a non-counital comonad with coproduct
 $\delta$, the latter has to be an $F$-module morphism,
in particular, $\delta F: GF\to GGF$ has to be an $F$-morphism 
 and this follows by commutativity of the rectangle in (\ref{m-lift-com})
provided the square in (\ref{m-lift-mon})  is commutative.   

To make the lifting $\wF$ a non-unital monad
with multiplication $\mu$, the latter has to be a $G$-comodule morphism, 
in particular, 
$\mu G: FFG\to FG$ has to be a $G$-module morphism and this   
follows by commutativity of the rectangle 
in (\ref{m-lift-mon}) provided the square in (\ref{m-lift-com}) is commutative. 
\etm

 \btm\label{mix-nat-trans} {\bf Natural transformations.} \em
 The data given in \ref{lift-func} allow for natural transformations
 $$\begin{array}{rl}
 \xi:& \xymatrix{ G\ar[r]^{\eta G} & FG \ar[r]^\omega & GF \ar[r]^{\ve F} & F },\\
 \wkappa: & 
 \xymatrix{ GF\ar[r]^{\eta GF}& FGF \ar[r]^{\omega F} & GFF \ar[r]^{G\mu}& GF},\\
 \wtau: & \xymatrix{ FG \ar[r]^{F\delta} & FGG \ar[r]^{\omega G} & GFG \ar[r]^{\ve FG} & FG, }
 \end{array}$$
 with the  properties 
 $$\begin{array}{rl}
  G\mu\dcirc \wkappa F= \wkappa\dcirc G\mu, & 
   \wtau G\dcirc F\delta = F\delta \dcirc \wtau, \\[+1mm] 
  \mu\dcirc \xi F=  \ve F\dcirc \wkappa, &
   \xi G\dcirc \delta = \wtau\dcirc \eta G. \end{array}$$
 \begin{rlist}
 \item {\em If the rectangle in (\ref{m-lift-mon}) is commutative, then $\wkappa$ is idempotent.}
 \item {\em If the rectangle in (\ref{m-lift-com}) is commutative, then $\wtau$ is idempotent.}
 \end{rlist}\etm
%Note that (i) is a special case of Lemma \ref{Apple-prop}(4); the proof 
%of (ii) is dual to that for (i).
\smallskip

To make the liftings weak comonads or weak monads, respectively,
we have to find pre-units or pre-counits, respectively. In what follows we
 consider these questions.

\btm\label {qu-coun}{\bf Lemma.} (Pre-counits for $\oG$) 
Assume the diagrams in {\rm (\ref{m-lift-mon})} to be commutative. Then
the following are equivalent:
\begin{blist}
\item for any $(A,\varphi)\in \rA_F$,  
      $\ve_A: G(A)\to A$ is an  $F$-module morphism;
\item $\ve F: GF\to F$ is an $F$-morphism;
\item $\vartheta= \mu \dcirc F\eta$ induces commutativity of the diagram
\begin{equation}\label{cond-ve}
\xymatrix{
  FG \ar[r]^{F\ve} \ar[d]_\omega & F \ar[d]^{\vartheta} \\ 
  GF \ar[r]^{\ve F} & F .}
\end{equation} 
\end{blist}

If these conditions are satisfied, then (with   $\gamma= G\ve \dcirc \vartheta$)
$$\mu G \dcirc F\wtau = \wtau\dcirc \mu G \quad \mbox{ and } \quad  
\wtau = \vartheta \gamma.$$  
\etm 
\begin{proof} This is shown by straightforward verification.
\end{proof}

\btm\label{reg-ve}{\bf Proposition.} 
 Assume the diagrams 
 in {\rm (\ref{m-lift-mon}), (\ref{m-lift-com})} and {\rm (\ref{cond-ve})}
 to be commutative. 
Then $(\oG,\delta,\ve)$ is a weak comonad on $\rA_F$.
\etm
\begin{proof} This follows from the preceding observations. 
\end{proof}
 
Dual to Lemma \ref{qu-coun} and \ref{reg-ve}
we obtain  for the quasi-units for $\wF$: 

\btm \label{unit-F}{\bf Lemma.} (Pre-units for $\wF$) 
Assume the diagrams in {\rm (\ref{m-lift-com})} to be commutative. Then
the following are equivalent:
\begin{blist}
\item for any $(A,\up)\in \rA^G$,  
      $\eta_A: A \to F(A)$ is a $G$-comodule morphism;
\item $\eta G: G\to FG$ is $G$-colinear;
\item $\gamma= G\ve\dcirc \delta$ induces commutativity of the diagram
\begin{equation}\label{eta-unit}
\xymatrix{ G\ar[d]_\gamma \ar[r]^{\eta G} & FG \ar[d]^\omega \\
  G \ar[r]^{G\eta} & GF . }
\end{equation} 
\end{blist}
If these conditions are satisfied, then 
$$G\wkappa \dcirc \delta F = \delta F \dcirc\wkappa \quad \mbox{ and }\quad 
 \wkappa = \gamma \vartheta . $$
\etm 

Summing up the above observations yields the

\btm\label{reg-eta}{\bf Proposition.}  
 Assume the diagrams in 
{\rm (\ref{m-lift-mon}), (\ref{m-lift-com})} and {\rm (\ref{eta-unit})}
 to be commutative. 
Then $(\wF,\mu,\eta)$ is a weak monad on $\rA^G$.
\etm

 One may consider alternative  choices for a pre-counit for $\oG$ or a pre-unit for $\wF$.

\btm\label{alt-qu-coun}{\bf Lemma.}  
Assume the diagrams in {\rm (\ref{m-lift-mon})} to be commutative.
With the notations from {\rm \ref{mix-nat-trans}}, the following are equivalent:
\begin{blist}
\item for any $(A,\varphi)\in \rA_F$,  
 $ \ove_A: \xymatrix{ 
 G(A) \ar[r]^{\xi_A}  & F(A) \ar[r]^\varphi & A}$ 
   is an $F$-module morphism;
\item $ \ove F: \xymatrix{GF\ar[r]^{\xi F} &FF \ar[r]^\mu & F}$ (=
$\xymatrix{ 
 GF \ar[r]^{\wkappa} & GF \ar[r]^{\ve F} & F}$) is an $F$-morphism; 
 
\item commutativity of the diagram
\begin{equation}\label{counit-2}
\xymatrix{  FFG \ar[r]^{F\omega} & 
        FGF \ar[r]^{F\ve F} & FF \ar[d]^\mu \\
    FG  \ar[u]^{F\eta G} \ar[r]^\omega & GF \ar[r]^{\ve F} &F .}  
\end{equation}
If these conditions are satisfied, then   
 $$ \wtau =\mu G\dcirc  F\wtau \dcirc F\eta G .$$ 
\end{blist}
\etm 
\begin{proof} The proof can be obtained by some diagram constructions.
\end{proof}

Notice that commutativity of (\ref{cond-ve}) implies commutativity of (\ref{counit-2}).

 \btm\label {alt-qu-un}{\bf Lemma.}   
Assume  the diagrams in {\rm (\ref{m-lift-com})} to be commutative. Then
the following are equivalent:
\begin{blist}
\item for any $(A,\up) \in \rA^G$, 
   $\widehat \eta:
\xymatrix{ A \ar[r]^{\up\quad} & G(A) \ar[r]^{\xi_A}  & F(A) } $ 
is a $G$-comodule morphism;
\item 
$\weta G: \xymatrix{G \ar[r]^{\eta G}& FG \ar[r]^{\wtau} & FG}$
$( = \xymatrix{G \ar[r]^{\delta}& GG \ar[r]^{\xi G} & FG})$
        is $G$-colinear;
\item commutativity of the diagram
\begin{equation}\label{unit-2}
\xymatrix{ G\ar[d]_\delta \ar[r]^{\eta G} & FG \ar[r]^\omega & GF \\
  GG \ar[r] ^{G\eta G} & GFG \ar[r]^{G\omega} & GGF \ar[u]_{G\ve F} . }
\end{equation}
\end{blist}

If these conditions are satisfied, then 
$$\wkappa = G\ve F \dcirc G\wkappa \dcirc \delta F.$$
\etm 
\begin{proof} The situation is dual to Lemma \ref{alt-qu-coun}. \end{proof}

Notice that commutativity of (\ref{eta-unit}) implies commutativity of (\ref{unit-2}).

\btm \label{oG}{\bf Proposition.} With the data given in {\rm \ref{lift-func}},
assume 
the diagrams in {\rm (\ref{m-lift-mon}), (\ref{m-lift-com})} 
and {\rm (\ref{counit-2})} to be commutative.  
\begin{zlist}
\item If {\rm (\ref{unit-2})} is commutative, then
$\ove$ from {\rm \ref{alt-qu-coun}} is regular for $\delta$, 
and for $\odelta:G\to GG$ with
$$  
\odelta F: \xymatrix{ 
GF\ar[r]^{\delta F} & GGF \ar[r]^{G\wkappa} & GGF,}
$$  
 $(\oG,\odelta,\ove)$ is an $r$-counital comonad on $\rA_F$.
 
\item If {\rm (\ref{eta-unit})} is commutative, then 
   $\odelta  F=  \delta F \dcirc\wkappa$ 
   and  $(\oG,\odelta,\ove)$ is a weak comonad on $\rA_F$.
\end{zlist}
\etm 
\begin{proof}  This can be shown by suitable diagram constructions. 
\end{proof}

\btm \label{wF}{\bf Proposition.} With the data given in {\rm \ref{lift-func}},
assume  
the diagrams in {\rm (\ref{m-lift-mon}), (\ref{m-lift-com})}, and 
{\rm (\ref{unit-2})} to be commutative.  
\begin{zlist}
\item If {\rm (\ref{counit-2})} is commutative, then
$\weta$ in {\rm \ref{alt-qu-un}} is  regular for $\mu$, and
for $\wmu:FF\to F$ with 
$$  
\wmu G: \xymatrix{  FFG \ar[r]^{F\wtau } & FFG \ar[r]^{\mu G} & FG,}
$$  
 $(\wF,\wmu,\weta)$ is an $r$-unital monad on $\rA^G$.
 
\item If {\rm (\ref{cond-ve})} is commutative, then 
  $\wmu G= \wtau\dcirc \mu G $ and   
   $(\wF,\wmu,\weta)$ is a weak monad on $\rA^G$.
\end{zlist}
\etm 
\begin{proof} This is dual to Proposition \ref{oG}. \end{proof}

\bigskip

{\bf Acknowledgments.} The author wants to thank Gabriella B\"ohm,
Tomasz Brzezi\'nski and Bachuki Mesablishvili for their interest in a previous
version of this paper 
and for helpful comments on the subject.

\end{document}